\documentclass[12pt]{amsart}

\usepackage{latexsym}
\usepackage{psfrag}
\usepackage{amsmath}
\usepackage{amssymb}
\usepackage{epsfig}
\usepackage{amsfonts}
\usepackage{amscd}
\usepackage{mathrsfs}
\usepackage{graphicx}
\usepackage{enumerate}

\newtheorem{theorem}{Theorem}%[section]

\newtheorem{corollary}[theorem]{Corollary}

\begin{document}

\title[A-G Series for Schur]{Andrews-Gordon Type Series for Schur's Partition Identity}

\author[Kur\c{s}ung\"{o}z]{Ka\u{g}an Kur\c{s}ung\"{o}z}
\address{Faculty of Engineering and Natural Sciences, Sabanc{\i} University, \.{I}stanbul, Turkey}
\email{kursungoz@sabanciuniv.edu}

%    For articles to be published after 1 January 2010, you may use
%    the following version:
\subjclass[2010]{05A17, 05A15, 11P84}

\keywords{integer partition, partition generating function, Andrews-Gordon identities, Schur's partition identity}

\date{December 2018}

\begin{abstract}
\noindent
  We construct an evidently positive multiple series as 
  a generating function for partitions satisfying the multiplicity condition in 
  Schur's partition theorem.  
  Refinements of the series when parts in the said partitions are classified according to 
  their parities or values mod 3 are also considered.  
  Direct combinatorial interpretations of the series are provided.  
\end{abstract}

\maketitle

\section{Introduction}
\label{secIntro}

Schur's partition identity~\cite{Schur} can be seen as an unexpected continuation 
of the line starting with Euler's partition identity~\cite{Euler}, 
and Rogers-Ramanujan identities~\cite{RR}.  
Euler's partition identity deals with partitions into 1-distinct parts, 
and Rogers-Ramanujan identities with partitions into 2-distinct part.  
Here, a partition into $d$-distinct means the pairwise difference of parts is at least $d$~\cite{AE}.

\begin{theorem}[Schur's partition identity]
\label{thmSchur}
  For any non-negative integer $n$, 
  the number of partitions of $n$ into parts that are 3-distinct 
  and that contain no consecutive multiples of 3 
  equals the number of partitions of $n$ into parts that are $\pm 1 \pmod{6}$.  
\end{theorem}

A partition of a positive integer $n$ is a non-decreasing sequence of non-negative integers 
$0 \leq \lambda_1 \leq \lambda_2 \leq \cdots \leq \lambda_k$ that sum up to $n$.  
$\lambda_i$'s are called parts, and $k$ is the number of parts.  
$n$ is called the weight of the partition.  
For instance, there are 5 partitions of 4 into positive parts:  
\begin{align*}
 4, \; 1+3, \; 2+2, \; 1+1+2, \; 1+1+1+1.  
\end{align*}
The only partition of zero is agreed to be the empty partition.  
Sometimes we will allow zeros to appear in partitions.  
Presence of zeros does not change the weight, but it changes the length.  
We will use partitions both into positive parts, and into non-negative parts 
throughout, but we will make it explicit when we allow zeros.  

If we call the former kind of partitions in Theorem \ref{thmSchur} $s(n)$, 
then the theorem is 
\begin{align}
\label{eqSchur}
  \sum_{n \geq 0} s(n) q^n = \frac{1}{ (q; q^6)_{\infty} (q^5; q^6)_{\infty} }
\end{align}

Here and elsewhere in the note, we use the standard $q$-series notations~\cite{TheBlueBook}
\begin{align*}
 (a; q)_n & = \prod_{j = 1}^n (1 - a q^{j-1}), \\
 (a; q)_{\infty} & = \lim_{n \to \infty} (a; q)_n = \prod_{j = 1}^\infty (1 - a q^{j-1}), 
\end{align*}
and a partition generating function is a formal series 
\begin{align*}
 \sum_{n \geq 0} a_n q^n
\end{align*}
where $a_n$ is the number of partitions of $n$ satisfying the given criteria.  

It is straightforward to see that the right-hand side of \eqref{eqSchur} 
generates partitions into parts that are $\pm 1 \pmod{6}$~\cite{AE}.  
The next challenge was to write a generating function for $s(n)$ 
with a direct combinatorial interpretation.  
By this, we mean an evidently positive series in the form of 
the left-hand side of, say, one of Andrews-Gordon identities~\cite{PNAS}
and its combinatorial interpretation~\cite{Bres, KK-parity}.  
\begin{align}
\label{eqAG}
  \sum_{n_1, \ldots, n_{k-1}} 
  \frac{ q^{ N_1^2 + \cdots + N_{k-1}^2 } }{ (q; q)_{n_1} \cdots (q; q)_{n_{k-1}} }
  = \prod_{ \substack{n \geq 1 \\ n \not\equiv 0, \pm k \pmod{2k+1}} } \frac{ 1 }{ (1 - q^n) }, 
\end{align}
where $N_r = n_{r} + n_{r+1} + \cdots + n_{k-1}$.  

One such generating function is given in~\cite{Alladi-Gordon} as 
\begin{align*}
  \sum_{m, n \geq 0} s(m, n) x^m q^n
  = \sum_{i, j, k \geq 0} 
    \frac{ q^{ \frac{3}{2}i^2 - \frac{7}{2}i + \frac{3}{2}j^2 - \frac{5}{2}j + 3k^2 - 3k 
      + 3(ij+ik+hk) } x^{i+j+k} }
    { (q^3; q^3)_i (q^3; q^3)_j (q^3; q^3)_k }
\end{align*}
where $s(m, n)$ is the number of partitions enumerated by $s(n)$ having $m$ parts.  

Recently, another series is given in~\cite{ABM}.  
\begin{align*}
  \sum_{m, n \geq 0} s(m, n) x^m q^n
  = \sum_{m, n \geq 0} 
    \frac{ (-1)^n q^{ (3n+m)^2 + \frac{m(m-1)}{2} } x^{m+2n} }
    { (q; q)_m (q^6; q^6)_n}
\end{align*}
However, it is not immediately clear that the last series generate positive coefficients for all $q^n$ 
upon expansion as a formal power series.  
The authors of~\cite{ABM} also remark the scarcity of series or multiple series generating functions 
for $s(n)$ in the literature.  

Our aim in this paper is to construct another such series for $s(m, n)$, 
and hence for $s(n)$ with a direct combinatorial interpretation in Section \ref{secMain}.  
The ideas are similar to those in~\cite{KK-Cpr}, 
but missing key parts as stated in~\cite{KK-Cpr} are discovered here.  
Then, in sections \ref{secMod2} and \ref{secMod3}, 
we refine the series to account for parts that fall in distinct residue classes 
modulo 2 and 3, respectively.  
We conclude with further open problems.  

\section{The Main Triple Series}
\label{secMain}

\begin{theorem}
\label{thmMain}
  For $n, m \in \mathbb{N}$, 
  let $s(m, n)$ denote the number of partitions of $n$ into $m$ parts 
  which are pairwise at least three apart, and multiples of 3 are at least 6 apart.  
  Then,
  \begin{align}
  \nonumber
    & \sum_{m, n \geq 0} s(m, n) x^m q^n\\
  \label{eqMainSeries}
    = & \sum_{ \substack{n_1 \geq 0 \\ n_{21}, n_{22} \geq 0} } \frac
      { q^{ 6 n_{21}^2 - n_{21} + 6 n_{22}^2 + n_{22} + 2 n_1^2 - n_1 
	    + 12 n_{21} n_{22} + 6 (n_{21} + n_{22})n_1 } 
	x^{ 2 n_{21} + 2 n_{22} + n_{1} } }
      { (q; q)_{n_{1}} (q^{6}; q^{6})_{n_{21}} (q^{6}; q^{6})_{n_{22}} }.  
  \end{align}
\end{theorem}

{\bf Remark: } Once the theorem is proven, it is easy to see that the series \eqref{eqMainSeries}
converges for $x = 1$ (and of course for $\vert q \vert < 1$)~\cite{A-Schur, ABM}.  

\begin{proof}
 Throughout the proof, $\lambda$ will denote a partition enumerated by $s(m, n)$.  
 Clearly, $\lambda$ has distinct parts, 
 the pairwise difference between parts is at least three, 
 and consecutive multiples of 3 cannot appear.  
 Equivalently, if there are successive parts which are exactly three apart, 
 they must be both $1 \pmod 3$ or $2 \pmod 3$.  
 Given $\lambda$, we find the maximum number of such pairs.  
 We indicate those pairs in brackets such as 
 \begin{align*}
  [3k+1, 3k+4], \textrm{ or } [3k+2, 3k+5],
 \end{align*}
 when expressing $\lambda$.  
 We call $[3k+1, 3k+4]$ a $1 \pmod 3$ pair, and $[3k+2, 3k+5]$ a $2 \pmod 3$ pair.  
 The only ambiguity arises when there is a streak of parts 
 \begin{align*}
  3k+1, 3k+4, \ldots, 3k+3s-2, \textrm{ or }
  3k+2, 3k+5, \ldots, 3k+3s-1.  
 \end{align*}
 In either case, we pair the leftmost parts first, and continue recursively for yet unbound parts.  
 There will be instances below where we use the rightmost pair instead of the leftmost one first 
 and continue with the rightmost unbound pair, 
 but the number of pairs will not change.  
 
 Once the pairing is complete, 
 we call the number of $1 \pmod 3$ pairs $n_{21}$, 
 the number of $2 \pmod 3$ pairs $n_{22}$, 
 and the number of the remaining parts, namely \emph{singletons}, $n_1$.  
 Singletons have difference at least four with either the preceding or the succeeding part, or both, 
 because otherwise there would have been one more pair.  
 
 We will show that such $\lambda$ corresponds to a unique quadruple 
 $(\beta, \mu, \eta^m, \eta^d)$, where 
 \begin{itemize}
 
  \item the base partition $\beta$ is enumerated by $s(m, n)$, 
    has $n_{21}$ $1 \pmod 3$ pairs, $n_{22}$ $2 \pmod 3$ pairs, $n_1$ singletons, 
    and minimum weight; 
  
  \item $\mu$ is a partition into $n_1$ parts (allowing zeros); 
  
  \item $\eta^m$ is a partition into $n_{21}$ multiples of 6 (allowing zeros); 
  
  \item $\eta^d$ is a partition into $n_{22}$ multiples of 6 (allowing zeros).  
  
 \end{itemize}
 
 We will transform $\lambda$ to $(\beta, \mu, \eta^m, \eta^d)$ 
 through a unique series of backward moves, 
 and $(\beta, \mu, \eta^m, \eta^d)$ to $\lambda$ through a unique series of forward moves.  
 The two series of moves will be inverses to each other.  
 
 Before describing these transformations, we discuss the moves.  
 They will be useful in the construction of the base partition $\beta$.  
 
 The forward or backward moves on a $1 \pmod 3$ pair is basically
 \begin{align}
 \label{moves1mod3pair}
  [\mathbf{3k+1, 3k+4}] 
  \begin{array}{c}
   \underrightarrow{ \textrm{one forward move} } \vspace{2mm} \\
   \overleftarrow{ \textrm{one backward move} }.  
  \end{array}
  [\mathbf{3k+4, 3k+7}]
 \end{align}
 Notice that both change the weight by 6.  
 We highlight the moved pair.  
 
 Of course, the pair may be immediately preceded or succeeded by other pairs or singletons.  
 \begin{align*}
  (\textrm{parts} \leq 3k-2), & [\mathbf{3k+1, 3k+4}], [3k+8, 3k+11], [3k+14, 3k+17], \ldots, \\
    & [3k+2+6s, 3k+5+6s], (\textrm{parts } \geq 3k+10+6s) 
 \end{align*}
 \begin{align*}
 & \bigg\downarrow \textrm{one forward move} 
 \end{align*}
 \begin{align*}
  (\textrm{parts} \leq 3k-2), & [\mathbf{3k+4}, \underbrace{\mathbf{3k+7}], [3k+8}_{!}, 3k+11], [3k+14, 3k+17], \ldots, \\
    & [3k+2+6s, 3k+5+6s], (\textrm{parts } \geq 3k+10+6s) 
 \end{align*}
 \begin{align*}
 & \bigg\downarrow \textrm{adjustment} 
 \end{align*}
 \begin{align*}
 (\textrm{parts} \leq 3k-2), & [3k+2, 3k+5], [\mathbf{3k+10}, \underbrace{\mathbf{3k+13}], [3k+14}_{!}, 3k+17], \ldots, \\
    & [3k+2+6s, 3k+5+6s], (\textrm{parts } \geq 3k+10+6s) 
 \end{align*}
 \begin{align*}
 & \bigg\downarrow \textrm{after } (s-1) \textrm{ similar adjustments} 
 \end{align*}
 \begin{align}
 \nonumber
  (\textrm{parts} \leq 3k-2), & [3k+2, 3k+5], [3k+8, 3k+11], \ldots, [3k-4+6s, 3k-1+6s], \\
 \label{movefwd1mod3pairadjust}
    & [\mathbf{3k+4+6s, 3k+7+6s}], (\textrm{parts } \geq 3k+10+6s), 
 \end{align}
 for $s \geq 1$.  
 The move increases the weight by 6.  
 The adjustments preserve the weight.  
 The inverse of the above operation is a backward move on the indicated $1 \pmod 3$ pair.  
 \begin{align*}
  (\textrm{parts} \leq 3k-2), & [3k+2, 3k+5], [3k+8, 3k+11], \ldots, [3k-4+6s, 3k-1+6s], \\
    & [\mathbf{3k+4+6s, 3k+7+6s}], (\textrm{parts } \geq 3k+10+6s), 
 \end{align*}
 \begin{align*}
 & \bigg\downarrow \textrm{one backward move} 
 \end{align*}
 \begin{align*}
  (\textrm{parts} \leq 3k-2), & [3k+2, 3k+5], [3k+8, 3k+11], \ldots, \\
    & [3k-4+6s, \underbrace{3k-1+6s], [\mathbf{3k+1+6s}}_{!}, \mathbf{3k+4+6s}], \\ 
    & (\textrm{parts } \geq 3k+10+6s), 
 \end{align*}
 \begin{align*}
 & \bigg\downarrow \textrm{adjustment} 
 \end{align*}
 \begin{align*}
  (\textrm{parts} \leq 3k-2), & [3k+2, 3k+5], [3k+8, 3k+11], \ldots, \\
    & [3k-10+6s, \underbrace{3k-7+6s], [\mathbf{3k-5+6s}}_{!}, \mathbf{3k-2+6s}], \\ 
    & [3k+2+6s, 3k+5+6s] (\textrm{parts } \geq 3k+10+6s), 
 \end{align*}
 \begin{align*}
 & \bigg\downarrow \textrm{after } (s-1) \textrm{ similar adjustments} 
 \end{align*}
 \begin{align}
 \nonumber
  (\textrm{parts} \leq 3k-2), & [\mathbf{3k+1, 3k+4}], [3k+8, 3k+11], [3k+14, 3k+17], \ldots, \\
  \label{movebackwd1mod3pairadjust}
    & [3k+2+6s, 3k+5+6s], (\textrm{parts } \geq 3k+10+6s).  
 \end{align}
 
%  In case \eqref{movefwd1mod3pairadjust}, if there is a $3k+9+6s$, 
%  so the larger parts are $\geq 3k+13+6s$, 
%  the final adjustment would have been: 
%  \begin{align*}
%   (\textrm{parts } \leq 3k-1+6s), & [\mathbf{ 3k+4+6s}, \underbrace{\mathbf{3k+7+6s }], 3k+9+6s}_{!}, \\
%   & (\textrm{parts } \geq 3k+13+6s)
%  \end{align*}
%  \begin{align*}
%  & \bigg\downarrow \textrm{adjustment} 
%  \end{align*}
%  \begin{align*}
%   (\textrm{parts } \leq 3k-1+6s), & 3k+3+6s, [\mathbf{3k+7+6s, 3k+10+6s}], \\ & (\textrm{parts } \geq 3k+13+6s).  
%  \end{align*}
%  
%  Again, in case \eqref{movefwd1mod3pairadjust}, if there is a $3k+8+6s$, 
%  so the remaining larger parts are $\geq 3k+12+6s$, 
  
 \begin{align*}
  (\textrm{parts } \leq 3k-2), [\mathbf{3k+1, 3k+4}], 3k+8, (\textrm{parts } \geq 3k+14)
 \end{align*}
 \begin{align*}
 & \bigg\downarrow \textrm{one forward move} 
 \end{align*}
 \begin{align*}
  (\textrm{parts } \leq 3k-2), [\mathbf{3k+4}, \underbrace{\mathbf{3k+7}], 3k+8}_{!}, (\textrm{parts } \geq 3k+14)
 \end{align*}
 \begin{align*}
 & \bigg\downarrow \textrm{adjustment} 
 \end{align*}
 \begin{align}
 \label{movefwd1mod3paircaseC1}
  (\textrm{parts } \leq 3k-2), 3k+2, [\mathbf{3k+7, 3k+10}], (\textrm{parts } \geq 3k+14).  
 \end{align}
 
 \begin{align*}
  (\textrm{parts } \leq 3k-2), [\mathbf{3k+1, 3k+4}], 3k+8, 3k+12, (\textrm{parts } \geq 3k+16)
 \end{align*}
 \begin{align*}
 & \bigg\downarrow \textrm{one forward move} 
 \end{align*}
 \begin{align*}
  (\textrm{parts } \leq 3k-2), [\mathbf{3k+4}, \underbrace{\mathbf{3k+7}], 3k+8}_{!}, 3k+12, (\textrm{parts } \geq 3k+16)
 \end{align*}
 \begin{align*}
 & \bigg\downarrow \textrm{adjustment} 
 \end{align*}
 \begin{align*}
  (\textrm{parts } \leq 3k-2), 3k+2 [\mathbf{3k+7}, \underbrace{\mathbf{3k+10}], 3k+12}_{!}, (\textrm{parts } \geq 3k+16)
 \end{align*}
 \begin{align*}
 & \bigg\downarrow \textrm{one more adjustment} 
 \end{align*}
 \begin{align}
 \label{movefwd1mod3paircaseC2}
  (\textrm{parts } \leq 3k-2), 3k+2, 3k+6, [\mathbf{3k+10, 3k+13}], (\textrm{parts } \geq 3k+16).  
 \end{align}
 
 In both cases \eqref{movefwd1mod3paircaseC1} and \eqref{movefwd1mod3paircaseC2}, 
 the forward move adds 6 to the weight of the partition, 
 and the adjustments do not change it.  
 The inverses of \eqref{movefwd1mod3paircaseC1} and \eqref{movefwd1mod3paircaseC2} 
 defining the corresponding backward moves are straightforward to describe.  
 
 Clearly, cases \eqref{movefwd1mod3pairadjust}, \eqref{movefwd1mod3paircaseC1} and \eqref{movefwd1mod3paircaseC2} 
 leave out some other possibilities.  
 Still, the remaining cases, along with their inverses, can be treated using combinations of adjustments 
 in \eqref{movefwd1mod3pairadjust}, \eqref{movefwd1mod3paircaseC1} and \eqref{movefwd1mod3paircaseC2}.  
 
 If we want to perform a forward move in 
 \begin{align*}
  (\textrm{parts } \leq 3k-2), [\mathbf{3k+1, 3k+4}], 3k+7, (\textrm{parts } \geq 3k+11), 
 \end{align*}
 we update the pairing first: 
 \begin{align}
 \label{movefwd1mod3pairupdate}
  (\textrm{parts } \leq 3k-2), 3k+1, [\mathbf{3k+4, 3k+7}], (\textrm{parts } \geq 3k+11), 
 \end{align}
 and then perform the forward move on the indicated pair.  
 
 When two $1 \pmod 3$ pairs immediately follow each other as in 
 \begin{align*}
  (\textrm{parts } \leq 3k-2), [3k+1, 3k+4], [3k+7, 3k+10], (\textrm{parts } \geq 3k+13), 
 \end{align*}
 forcing the smaller pair to move forward with some adjustments to follow 
 will result in the appearance that the larger pair has been moved instead.  
 To avoid this, we simply forbid forward moves on $1 \pmod 3$ pairs 
 immediately preceding another.  
 In general, we disallow any forward move causing 
 \begin{align*}
  [3k+4, \underbrace{3k+7], [3k+7}_{!}, 3k+10]
 \end{align*}
 either immediately or after adjustments.  
 
 One can easily show that a forward move on the larger of the immediately succeding $1 \pmod 3$ pairs 
 allows one forward move on the smaller pair.  
 
 The same phenomenon is observed for backward moves with the smaller and larger pairs swapped.  
 
 The above discussion can be suitably adjusted to describe forward and backward moves of $2 \pmod 3$ pairs 
 in all possible cases.  
 
 The moves on singletons are merely adding or subtracting 1's, 
 thus altering the weight by 1 or $-1$, respectively.  
 The distance between two singletons is at least 4, 
 so when the distance is exactly 4 as in 
 \begin{align*}
  (\textrm{parts } \leq k-3), k, k+4, (\textrm{parts } \geq k+7), 
 \end{align*}
 the smaller singleton cannot be moved forward, 
 and the larger singleton cannot be moved backward.  
 
 In some cases we encounter, 
 we will need to move the singletons through pairs as follows.  
 \begin{align*}
  (\textrm{parts } \leq 3k-3), & \mathbf{3k+1}, [3k+4, 3k+7], [3k+10, 3k+13], \\ 
  & \ldots, [3k-2+6s, 3k+1+6s], (\textrm{parts } \geq 3k+5+6s)
 \end{align*}
 \begin{align*}
  \bigg\downarrow \textrm{ regroup the pairs }
 \end{align*}
 \begin{align*}
  & (\textrm{parts } \leq 3k-3), [3k+1, 3k+4], [3k+7, 3k+10], \\ 
  & \ldots, [3k-5+6s, 3k-2+6s], \mathbf{3k+1+6s}, (\textrm{parts } \geq 3k+5+6s)
 \end{align*}
 \begin{align*}
  \bigg\downarrow \textrm{ one forward move }
 \end{align*}
 \begin{align}
 \nonumber
  & (\textrm{parts } \leq 3k-3), [3k+1, 3k+4], [3k+7, 3k+10], \\ 
 \label{movefwdsingletonthrupairs}
  & \ldots, [3k-5+6s, 3k-2+6s], \mathbf{3k+2+6s}, (\textrm{parts } \geq 3k+5+6s)
 \end{align}
 for $s \geq 1$, with another possible regrouping afterwards.  
 The reverse process, 
 i.e. moving the highlighted singleton through the pairs is easily defined.  
 Of course, if there was another singleton $=3k+4+6s$, 
 $3k+1$ could not have been a singleton in the first place 
 because the determination of pairs would have been incorrect.  
 Taking into account the construction of the base partition $\beta$, 
 and the correspondence between $\lambda$ and $(\beta, \mu, \eta^m, \eta^d)$ below, 
 one can show that this pathology can never occur.  
 
 Moving singletons through $2 \mod 3$ pairs is defined likewise.  
 
 Now that we defined all possible moves, 
 we can proceed with the construction of the base partition $\beta$.  
 Recall that $\beta$ is the partition conforming to the constraints 
 set forth by $s(m, n)$, with $n_{21}$ $1 \pmod 3$ pairs, $n_{22}$ $2 \pmod 3$ pairs, 
 $n_1$ singletons, and minimum possible weight.  
 
 Set $n_2 = n_{21} + n_{22}$, and for a moment 
 suppose that $\beta_0$ has $n_2$ $1 \pmod 3$ pairs, no $2 \pmod 3$ pairs, and $n_1$ singletons.  
 In other words, $\beta_0$ consists of $1 \pmod 3$ pairs, and singletons only.  
 We will argue that $\beta_0$ can only be 
 \begin{align}
 \label{primitivebaseptnmain}
  \beta_0 = [1, 4], [7, 10], \ldots, [6n_2-5, 6n_2-2], 
  6n_2+1, 6n_2+5, \ldots, 6n_2+4n_1-3.  
 \end{align}
 Indeed, no further backward moves are possible either on the smallest pair $[1, 4]$ 
 or on the smallest singleton $6n_2+1$, 
 and both the pairs and singletons are tightly packed.  
 
 Now shift the largest $n_{22}$ of the $1 \pmod 3$ pairs as 
 \begin{align*}
  [3k+1, 3k+4] \rightarrow [3k+2, 3k+5]
 \end{align*}
 and do the necessary adjustments.  
 \begin{align*}
  & [1, 4], [7, 10], \ldots, [6n_{21}-5, 6n_{21}-2], 
  [6n_{21}+2, 6n_{21}+5], [6n_{21}+8, 6n_{21}+11], \\
  & \ldots, [6n_{21}+6n_{22}-4, \underbrace{6n_{21}+6n_{22}-1], 
  6n_{21}+6n_{22}+1}_{!}, 6n_{21}+6n_{22}+5, \\ & 
  \ldots, 6n_{21}+6n_{22}+4n_1-3
 \end{align*}
 \begin{align*}
  \bigg\downarrow n_{22} \textrm{ adjustments if } n_1 > 0
 \end{align*}
 \begin{align}
 \nonumber
  \beta = & [1, 4], [7, 10], \ldots, [6n_{21}-5, 6n_{21}-2], 6n_{21}+1, 
  [6n_{21}+5, 6n_{21}+8], \\
 \nonumber
  & [6n_{21}+11, 6n_{21}+14], \ldots, [6n_{21}+6n_{22}-1, 6n_{21}+6n_{22}+2], 
  6n_{21}+6n_{22}+5, \\ 
 \label{baseptnmain}
  & 6n_{21}+6n_{22}+9, \ldots, 6n_{21}+6n_{22}+4n_1-3.  
 \end{align}
 
 Again, one can check that both the pairs, $1 \pmod 3$ and $2 \pmod 3$ alike, 
 and the singletons are tightly packed.  
 No pair or singleton can be moved further backward.  
 Therefore, \eqref{baseptnmain} is the unique base partition $\beta$ 
 we have been looking for.  
 
 The partition $\beta_0$ has weight
 \begin{align*}
  \sum_{j = 1}^{n_2} (12 j - 7) + \sum_{j = 1}^{n_1} (6n_2 + 4j - 3)
  = 6n_2^2 - n_2 + 6n_2 n_1 + 2n_1^2 - n_1.  
 \end{align*}
 Recall that we set $n_2 = n_{21} + n_{22}$, 
 and in passing from \eqref{primitivebaseptnmain} to \eqref{baseptnmain}, 
 we performed $n_{22}$ shifts each of which adds 2 to the weight.  
 Thus, $\beta$ in \eqref{baseptnmain} has weight
 \begin{align*}
  & 6(n_{21} + n_{22})^2 - (n_{21}+n_{22}) + 2 n_{22} + 6 (n_{21} + n_{22})n_1 \\
  & = 6 n_{21}^2 - n_{21} + 6 n_{22}^2 + n_{22} + 2n_1^2 - n_1
    + 12 n_{21} n_{22} + 6 (n_{21} + n_{22}) n_1.  
 \end{align*}
 Clearly, $\beta$ has $2n_{21} + 2n_{22} + n_1$ parts.  
 $\mu$, $\eta^m$, $\eta^d$ are generated by 
 \begin{align*}
  \frac{1}{(q; q)_{n_1}}, \qquad
  \frac{1}{(q^6; q^6)_{n_{21}}}, \qquad
  \frac{1}{(q^6; q^6)_{n_{22}}}, \qquad
 \end{align*}
 respectively.  
 
 We have shown that 
 \begin{align*}
  & \sum_{(\beta, \mu, \eta^m, \eta^d)} 
    q^{ \vert \beta \vert + \vert \mu \vert + \vert \eta^m \vert + \vert \eta^d \vert } 
    x^{l(\beta)} \\
  & = \sum_{ \substack{n_1 \geq 0 \\ n_{21}, n_{22} \geq 0} } \frac
    { q^{ 6 n_{21}^2 - n_{21} + 6 n_{22}^2 + n_{22} + 2 n_1^2 - n_1 
	  + 12 n_{21} n_{22}  + 6 (n_{21} + n_{22})n_1 } 
      x^{ 2 n_{21} + 2 n_{22} + n_{1} } }
    { (q; q)_{n_{1}} (q^{6}; q^{6})_{n_{21}} (q^{6}; q^{6})_{n_{22}} }.  
 \end{align*}
 The proof will be complete once we establish the correspondence 
 between $\lambda$ and $(\beta, \mu, \eta^m, \eta^d)$, 
 so that 
 \begin{align*}
  \sum_{m, n \geq 0} s(m, n) x^m q^n
  = \sum_{\lambda} q^{\vert \lambda \vert} x^{l(\lambda)}
  = \sum_{(\beta, \mu, \eta^m, \eta^d)} 
    q^{ \vert \beta \vert + \vert \mu \vert + \vert \eta^m \vert + \vert \eta^d \vert } 
    x^{l(\beta)}.  
 \end{align*}
 
 Given $\lambda$ enumerated by $s(m, n)$, 
 we first determine the unique numbers $n_{21}$ of $1 \pmod 3$ pairs, 
 $n_{22}$ of $2 \pmod 3$ pairs, and $n_1$ of singletons, 
 so that $m = 2n_{21} + 2n_{22} + n_1$.  
 We move the smallest $1 \pmod 6$ pair $\frac{1}{6}\eta^m_1$ times backward, 
 so that it becomes $[1, 4]$.  
 If the smallest $1 \pmod 3$ pair is already $[1, 4]$, 
 we set $\frac{1}{6}\eta^m_1 = 0$.  
 We continue moving the $s$th smallest $1 \pmod 3$ pair backward 
 so that it becomes $[6s-5, 6s-2]$, 
 and recording the respective number of backward moves as $\frac{1}{6}\eta^m_s$, 
 for $s = 2, 3, \ldots, n_{21}$.  
 Because a backward move on a $1 \pmod 3$ pair 
 allows a backward move on the succeding pair, 
 \begin{align*}
  0 \leq \eta^m_1 \leq \eta^m_2 \leq \cdots \leq \eta^m_{n_ {21}}.  
 \end{align*}
 Parts of $\eta^m$ are all multiples of 6.   
 Also, as the weight of $\lambda$ decreases 6 by 6, 
 the weight of $\eta^m$ increases by the same amount.  
 The intermediate partition looks like
 \begin{align}
 \nonumber
  & [1, 4], [7, 10], \ldots, [6n_{21}-5, 6n_{21}-2], \\ 
 \label{intermptnfromlambda1}
  & (\textrm{parts } \geq 6n_{21} + 1, 
  \textrm{ all } 2 \pmod 3 \textrm{ pairs or singletons}).  
 \end{align}
 It is possible that $n_{21} = 0$, in which case one skips this phase.  
 
 Next, we move the $s$th smallest $2 \pmod 3$ pair as far as possible, 
 and record the number of backward moves as $\frac{1}{6}\eta^d_s$, 
 for $s = 1, 2, \ldots, n_{22}$, in the given order.  
 This determines the partition $\eta^d$ into $n_{22}$ multiples of 6.  
 The details are similar to the $1 \pmod 3$ case above.  
 The intermediate partition looks like
 \begin{align}
 \nonumber
  & [1, 4], [7, 10], \ldots, [6n_{21} - 5, 6n_{21} - 2], 
  [6n_{21} + 2, 6n_{21} + 5], [6n_{21} + 8, 6n_{21} + 11], \\ 
 \label{intermptnbackwdallpairsstowedcaseI}
  & \ldots, [6n_{21} + 6n_{22} - 4, 6n_{21} + 6n_{22} - 1], 
  (\textrm{ singletons } \geq 6n_{21} + 6n_{22} + 2), 
 \end{align}
 if the intermediate partition \eqref{intermptnfromlambda1} 
 contains no singleton $ = 6n_{21} + 1$, and like 
 \begin{align}
 \nonumber
  & [1, 4], [7, 10], \ldots, [6n_{21} - 5, 6n_{21} - 2], 6n_{21} + 1, \\
 \nonumber
  & [6n_{21} + 5, 6n_{21} + 8], [6n_{21} + 11, 6n_{21} + 14], \ldots, 
  [6n_{21} + 6n_{22} - 1, 6n_{21} + 6n_{22} + 2], \\
 \label{intermptnbackwdallpairsstowedcaseII}
  & (\textrm{ singletons } \geq 6n_{21} + 6n_{22} + 5), 
 \end{align}
 if the intermediate partition \eqref{intermptnfromlambda1} 
 contains a singleton $ = 6n_{21} + 1$.  
 Likewise, we skip this phase if $n_{22} = 0$.  
 In either \eqref{intermptnbackwdallpairsstowedcaseI} or \eqref{intermptnbackwdallpairsstowedcaseII}, 
 the singletons are at least 4 apart.  
 
 We finally move the $s$th smallest singleton backward $\mu_s$ times 
 for $s = 1, 2, \ldots, n_1$ to construct $\mu$ and to obtain the base partition $\beta$ 
 as in \eqref{baseptnmain}.  
 
 The above prrocedure is clearly uniquely reversible, 
 which will yield $\lambda$ from $(\beta, \mu, \eta^m, \eta^d)$ by a series of forward moves.  
 This concludes the proof.  
\end{proof}

{\bf Example: } Following the notation in the proof of Theorem \ref{thmMain}, 
we will construct $\lambda$ from the base partition $\beta$ where 
\begin{align*}
 n_{21} = 2, \; n_{22} = 2, \; n_1 = 1, \; 
 \mu = 2, \; \eta^m = 6+6, \; \textrm{ and } \eta^d = 0+6.  
\end{align*}
\begin{align*}
 \beta = [1, 4], [7, 10], \mathbf{13}, [17, 20], [23, 26]
\end{align*}
$\beta$ has weight
\begin{align*}
 6 n_{21}^2 - n_{21} + 6 n_{22}^2 + n_{22} + 2 n_1^2 - n_1 
 + 12 n_{21} n_{22} + 6 (n_{21} + n_{22})n_1
 = 121.  
\end{align*}
\begin{align*}
 \Bigg\downarrow \textrm{ one forward move }
\end{align*}
\begin{align*}
 \beta = [1, 4], [7, 10], \mathbf{14}, [17, 20], [23, 26]
\end{align*}
\begin{align*}
 \Bigg\downarrow \textrm{ regroup pairs }
\end{align*}
\begin{align*}
 \beta = [1, 4], [7, 10], [14, 17], [20, 23], \mathbf{26}
\end{align*}
\begin{align*}
 \Bigg\downarrow \textrm{ another forward move }
\end{align*}
\begin{align*}
 \beta = [1, 4], [7, 10], [14, 17], [\mathbf{20}, \mathbf{23}], 27
\end{align*}
Here is the end of the implementation of $\mu = 2$, which is a partition into $n_1 = 1$ part, 
allowing zeros.  We will continue with the incorporation of $\eta^d = 0+6$ 
as forward moves on the $2 \pmod 3$ pairs.  
\begin{align*}
 \Bigg\downarrow \textrm{ one forward move }
\end{align*}
\begin{align*}
 \beta = [1, 4], [7, 10], [14, 17], [\mathbf{23}, \underbrace{\mathbf{26}], 27}_{!}
\end{align*}
\begin{align*}
 \Bigg\downarrow \textrm{ adjustment }
\end{align*}
\begin{align*}
 \beta = [1, 4], [\mathbf{7}, \mathbf{10}], [14, 17], 21, [26, 29]
\end{align*}
We top it off with the implementation of $\eta^m = 6+6$ as forward moves on the $1 \pmod{3}$ pairs.  
\begin{align*}
 \Bigg\downarrow \textrm{ one forward move }
\end{align*}
\begin{align*}
 \beta = [1, 4], [\mathbf{10}, \underbrace{\mathbf{13}], [14}_{!}, 17], 21, [26, 29]
\end{align*}
\begin{align*}
 \Bigg\downarrow \textrm{ adjustment }
\end{align*}
\begin{align*}
 \beta = [1, 4], [8, 11], [\mathbf{16}, \underbrace{\mathbf{19}], 21}_{!}, [26, 29]
\end{align*}
\begin{align*}
 \Bigg\downarrow \textrm{ adjustment }
\end{align*}
\begin{align*}
 \beta = [\mathbf{1}, \mathbf{4}], [8, 11], 15, [19, 22], [26, 29]
\end{align*}
\begin{align*}
 \Bigg\downarrow \textrm{ one forward move }
\end{align*}
\begin{align*}
 \beta = [\mathbf{4}, \underbrace{\mathbf{7}], [8}_{!}, 11], 15, [19, 22], [26, 29]
\end{align*}
\begin{align*}
 \Bigg\downarrow \textrm{ adjustment }
\end{align*}
\begin{align*}
 \beta = [2, 5], [\mathbf{10}, \underbrace{\mathbf{13}], 15}_{!}, [19, 22], [26, 29]
\end{align*}
\begin{align*}
 \Bigg\downarrow \textrm{ adjustment }
\end{align*}
\begin{align*}
 \beta = [2, 5], 9, [13, 16], [19, 22], [26, 29]
\end{align*}
In the final configuration, 
\begin{align*}
 141 = \vert \lambda \vert 
 = \vert \beta \vert + \vert \mu \vert + \vert \eta^d \vert + \vert \eta^m \vert 
 = 121 + 2 + 6 + 12, 
\end{align*}
as claimed.  

\begin{corollary}
\label{corSchur}
  \begin{align}
  \label{eqCorSchur}
    \sum_{ \substack{n_1 \geq 0 \\ n_{21}, n_{22} \geq 0} } \frac
      { q^{ 6 n_{21}^2 - n_{21} + 6 n_{22}^2 + n_{22} + 2 n_1^2 - n_1 
	    + 12 n_{21} n_{22} + 6 (n_{21} + n_{22})n_1 } }
      { (q; q)_{n_{1}} (q^{6}; q^{6})_{n_{21}} (q^{6}; q^{6})_{n_{22}} }
    = \frac{1}{ (q; q^6)_\infty (q^5; q^6)_\infty }
  \end{align}
\end{corollary}

\begin{proof}
 By Theorem \ref{thmMain}, the series on the left hand side of \eqref{eqCorSchur} is 
 \begin{align*}
  \sum_{n \geq 0} \left( \sum_{m \geq 0} s(m, n) \right) q^n
  = \sum_{n \geq 0} s(n) q^n.  
 \end{align*}
 The Corollary follows by Schur's theorem~\cite{Schur}.  
\end{proof}

\section{Counting Parts According to Their Parities}
\label{secMod2}

In this section, we will refine Theorem \ref{thmMain} 
and obtain a multiple series generating function for partitions enumerated by $s(m, n)$
where odd parts and even parts are accounted for separately.  

\begin{theorem}
\label{thmMod2}
  For $m_1, m_0, n \geq 0$, 
  let $sp(m_1, m_0, n)$ be the number of partitions satisfying the conditions of $s(m_1+m_0, n)$ 
  such that $m_1$ of the parts are odd and $m_0$ are even.  Then, 
  \begin{align}
  \nonumber
    & \sum_{m_1, m_0, n \geq 0} sp(m_1, m_0, n) a^{m_1} b^{m_0} q^{n}\\
    \nonumber
    = & \sum_{ \substack{n_{11}, n_{10} \geq 0 \\ n_{21}, n_{22} \geq 0} } \frac
      { q^{ 6 n_{21}^2 - n_{21} + 6 n_{22}^2 + n_{22} + 2 n_{11}^2 - n_{11} + 2 n_{10}^2 } }
      { (q^2; q^2)_{n_{11}} (q^2; q^2)_{n_{10}} (q^{6}; q^{6})_{n_{21}} (q^{6}; q^{6})_{n_{22}} }  \\
      & \times q^{12 n_{21} n_{22} + 6 (n_{21} + n_{22})(n_{11} + n_{10}) + 4 n_{11} n_{10}}
      a^{ n_{21} + n_{22} + n_{11} } 
      b^{ n_{21} + n_{22} + n_{10} }.  
  \label{eqMod2Series}
  \end{align}
\end{theorem}

\begin{proof}
  We will base the construction on the proof of Theorem \ref{thmMain}.  
  The moves of the 1 or 2 $\pmod 3$ pairs through pairs of the other kind 
  or through singletons will be the same.  
  Also, in each pair, there is exactly one odd and one even part.  
  The difference here is that we need to differentiate the odd and even singletons, 
  and determine how to move them.  
  
  We will be moving singletons 2 times forward (double forward move) 
  or 2 times backward (double backward move), 
  therefore their parity does not change.  
 \begin{align*}
  (\textrm{ parts } \leq k-3), & \mathbf{k}, (\textrm{ parts } \geq k+5) \\
  \textrm{ a double forward move } 
  & \bigg\downarrow \bigg\uparrow 
  \textrm{ a double backward move }
  \\
  (\textrm{ parts } \leq k-3), & \mathbf{k+2}, (\textrm{ parts } \geq k+5).  
 \end{align*}
 
 If a singleton needs to pass through one or more pairs, 
 or when a singleton is tackled by a pair, 
 the singleton is shifted by 6, so that its parity is retained.  
 
 When a move requires adjustments to follow 
 which alters the number of pairs or singletons, 
 we simply do not allow that move.  
 For example, a double backward move on the singleton 8 in the configuration below 
 reduces the number of singletons in the end.  
 \begin{align*}
  [1, 4], \mathbf{8} \quad 
  \underrightarrow{\textrm{a double backward move}} \quad 
  [1, \underbrace{4], 6}_{!} \quad 
  \underrightarrow{\textrm{adjustment}} \quad 
  \mathbf{0}, [4, 7].  
 \end{align*}
 Although $[1, 4], 8$ is not a base partition in the sense of proof of Theorem \ref{thmMain}, 
 it will be a base partition for the purposes of this proof.  
 
 The only case that needs discussed is the two singletons of opposite parities 
 coming \emph{too close} to each other.  
 \begin{align*}
  (\textrm{ parts } \leq k-3), \mathbf{k}, k+5, k+9, \ldots, k+4s+1, 
  (\textrm{ parts } \geq k+4s+6)
 \end{align*}
 \begin{align*}
  \bigg\downarrow \textrm{ a double forward move }
 \end{align*}
 \begin{align*}
  (\textrm{ parts } \leq k-3), \underbrace{\mathbf{k+2}, k+5}_{!}, k+9, \ldots, k+4s+1, 
  (\textrm{ parts } \geq k+4s+6)
 \end{align*}
 This configuration may not be ruled out by the conditions set forth by $s(m, n)$, 
 however it certainly changes the number of singletons and the number of pairs 
 determined prior to it.  Therefore, it must be fixed.  
 \begin{align*}
  \bigg\downarrow \textrm{ adjustment }
 \end{align*}
 \begin{align*}
  (\textrm{ parts } \leq k-3), k+1, \underbrace{\mathbf{k+6}, k+9}_{!}, \ldots, k+4s+1, 
  (\textrm{ parts } \geq k+4s+6)
 \end{align*}
 \begin{align*}
  \bigg\downarrow \textrm{ after } (s-1) \textrm{ similar adjustments }
 \end{align*}
 \begin{align*}
  (\textrm{ parts } \leq k-3), k+1, k+5, k+9, \ldots, k+4s-3, \mathbf{ k+4s+2 }, 
  (\textrm{ parts } \geq k+4s+6)
 \end{align*}
 There may be other cases depending on the presence of pairs 
 immediately succeding $k+4s+1$, 
 but the above constellation is the main difference 
 from the construction in the proof of Theorem \ref{thmMain}.  
 
 Of course, we cannot double move forward the singletons 
 immediately preceding the singletons of the same parity.  
 Yet, when the difference is exactly 4 between two singletons, 
 a double forward move on the larger one allows one double move forward on the smaller.  
 
 The above discussion can be suitably adjusted for double backward moves.  
 
 To find the base partition with $n_{21}$ $1 \pmod 3$ pairs, 
 $n_{22}$ $2 \pmod 3$ pairs, $n_{11}$ odd singletons and $n_{10}$ even singletons, 
 we start with the base partition $\beta$ in \eqref{baseptnmain}.  
 It has $n_{21}$ $1 \pmod 3$ pairs, $n_{22}$ $2 \pmod 3$ pairs, 
 and $n_1 = n_{11} + n_{10}$ singletons.  
 As of now, all singletons are odd.  
 Then, we move the $n_{10}$ largest singletons once forward 
 to get the base partition
 \begin{align}
 \nonumber
 \beta = & [1, 4], [7, 10], \ldots, [6n_{21} - 5, 6n_{21} - 2], 6n_{21}+1, \\
 \nonumber
 & [6n_{21} + 5, 6n_{21} + 8], [6n_{21} + 11, 6n_{21} + 14], 
 \ldots, [6n_{21} + 6n_{22} - 1, 6n_{21} + 6n_{22} + 2], \\
 \nonumber
 & 6n_{21} + 6n_{22} + 5, 6n_{21} + 6n_{22} + 9, \ldots, 6n_{21} + 6n_{22} + 4n_{11} + 1, \\
 \nonumber
 & 6n_{21} + 6n_{22} + 4n_{11} + 6, 6n_{21} + 6n_{22} + 4n_{11} + 10, \\
 \label{baseptnMod2}
 & \ldots, 6n_{21} + 6n_{22} + 4n_{11} + 4n_{10} + 2.  
 \end{align}
 Thanks to the proof of Theorem \ref{thmMain}, 
 $\beta$ in \eqref{baseptnMod2} has weight
 \begin{align*}
  & 6 n_{21}^2 - n_{21} + 6 n_{22}^2 + n_{22} + 2 (n_{11} + n_{10})^2 - (n_{11} + n_{10}) \\
  & + 12 n_{21} n_{22} + 6 (n_{21} + n_{22}) (n_{11} + n_{10}) + n_{10} \\
  = & 6 n_{21}^2 - n_{21} + 6 n_{22}^2 + n_{22} + 2 n_{11}^2 - n_{11} + 2 n_{10}^2 \\
  & + 12 n_{21} n_{22} + 6 (n_{21} + n_{22}) (n_{11} + n_{10}) + 4 n_{10} n_{11}, 
 \end{align*}
 hence the exponent of $q$ in the numerator of the term in the multiple series 
 on the right hand side of \eqref{eqMod2Series}.  
 Clearly, no further double backward moves on the singletons are possible.  
 
 $\beta$ in \eqref{baseptnMod2} has $n_{21} + n_{22} + n_{11}$ odd parts, 
 hence the exponent of $a$ on the right hand side of \eqref{eqMod2Series}; 
 and it has $n_{21} + n_{22} + n_{10}$ even parts, 
 hence the exponent of $b$ in the same place.  
 The double forward moves on the odd or even singletons are in 1-1 correspondence with 
 partitions into $n_{11}$ or $n_{10}$ even parts (zeros allowed), respectively, 
 hence the factor $(q^2; q^2)_{n_{11}}(q^2; q^2)_{n_{10}}$ in the denominator 
 on the right hand side of \eqref{eqMod2Series}.  
\end{proof}

{\bf Example: } In this example, 
we will focus on the double forward of backward moves on singletons.  
The movement of pairs is the same as in the proof of Theorem \ref{thmMain}.  

We will recover the base partition $\beta$, $\mu^1$, $\mu^0$, $\eta^m$, and $\eta^d$ 
from $\lambda = 8 + 13 + 19 + 24$ for which $n_{11} = n_{10} = 2$, $n_{21} = n_{22} = 0$.  
Clearly, $\eta^m$ and $\eta^d$ are both empty partitions.  
We perform the double backward moves on the odd singletons first.  
\begin{align*}
 \lambda = 8, \mathbf{13}, 19, 24 
 \quad \underrightarrow{\textrm{one double backward move}} 
 \quad \underbrace{8, \mathbf{11}}_{!}, 19, 24 
\end{align*}
Notice that there are no problems as far as the conditions of $s(m,n)$ are concerned, 
but letting this intermediate partition stay in this form creates a $2 \pmod{3}$ pair, 
and changes $n_{22}$ and $n_{11}$.  
\begin{align*}
 \underrightarrow{\textrm{adjustment}} 
 \quad \mathbf{7}, 12, 19, 24 
 \quad \underrightarrow{\textrm{three more double backward moves}} 
 \quad 1, 12, \mathbf{19}, 24 
\end{align*}
We deduce that $\mu^1_1 = 8$.  
\begin{align*}
 \underrightarrow{\textrm{two double backward moves}} 
 \quad 1, \underbrace{12, \mathbf{15}}_{!}, 24
 \quad \underrightarrow{\textrm{adjustment}} 
 \quad 1, \mathbf{11}, 16, 24
\end{align*}
The same remark as above applies for the last adjustment.  
\begin{align*}
 \underrightarrow{\textrm{three more double backward moves}} 
 \quad 1, 5, \mathbf{16}, 24
\end{align*}
We now have $\mu^1 = 8+10$.  
At this point, it is easier to recover $\mu^0 = 6 + 10$, 
arriving at 
\begin{align*}
 \beta = 1, 5, 10, 14, 
\end{align*}
and 
\begin{align*}
 30 + 18 + 16 + 0 + 0 
 = \vert \beta \vert + \vert \mu^1 \vert + \vert \mu^0 \vert + \vert \eta^m \vert + \vert \eta^d \vert 
 = \vert \lambda \vert 
 = 64, 
\end{align*}
as Theorem \ref{thmMod2} claims.  

\section{Counting Parts According to Their Values Modulo 3}
\label{secMod3}

In this section, we will refine Theorem \ref{thmMain} 
and obtain a multiple series generating function for partitions enumerated by $s(m, n)$
where parts that are 1, 2, or 0 $\pmod 3$ are accounted for separately.  

\begin{theorem}
\label{thmMod3}
  For $m_1, m_2, m_0, n \geq 0$, 
  let $st(m_1, m_2, m_0, n)$ be the number of partitions satisfying the conditions of $s(m_1+m_2+m_0, n)$ 
  such that $m_i$ of the parts are $\equiv i \pmod 3$ for $i = 1, 2, 0$.  Then, 
  \begin{align}
  \nonumber
    & \sum_{m_1, m_2, m_0, n \geq 0} sp(m_1, m_2, m_0, n) a^{m_1} b^{m_2} c^{m_0} q^{n}\\
    \nonumber
    = & \sum_{ \substack{n_{11}, n_{12}, n_{10} \geq 0 \\ n_{21}, n_{22} \geq 0} } \frac
      { q^{ 6 n_{21}^2 - n_{21} + 6 n_{22}^2 + n_{22} + 3 n_{11}^2 - 2 n_{11} + 3 n_{12}^2 - n_{12} + 3 n_{10}^2 } }
      { (q^3; q^3)_{n_{11}} (q^3; q^3)_{n_{12}} (q^3; q^3)_{n_{10}} (q^{6}; q^{6})_{n_{21}} (q^{6}; q^{6})_{n_{22}} }  \\
      \nonumber
      & \times q^{ 12 n_{21} n_{22} + 6 (n_{21} + n_{22})(n_{11} + n_{12} + n_{10})
		  + 3 n_{11} n_{12} + 3 n_{11} n_{10} + 3 n_{12} n_{10}} \\      
      & \times a^{ 2 n_{21} + n_{11} } 
      b^{ 2 n_{22} + n_{12} } 
      c^{ n_{10} }.  
  \label{eqMod3Series}
  \end{align}
\end{theorem}

\begin{proof}
  The method of this proof will be similar to the proof of Theorem \ref{thmMod2} 
  based on Theorem \ref{thmMain}.  
  The triple forward moves on singletons will preserve their value modulo 3 etc.  
  We leave the details of the triple forward or backward moves 
  bringing singletons in different residue classes $\pmod{3}$ 
  too close together to the reader.  
  
  The novelty here is the construction of the base partition $\beta$.  
  For a moment, assume that there are no pairs, 
  $n_{11}$ $1 \pmod{3}$ singletons, $n_{12}$ $2 \pmod{3}$ singletons, 
  and $n_{10}$ $0 \pmod{3}$ singletons.  
  We line these up as follows: 
  \begin{align}
  \nonumber
    & \{1, 7, \ldots, 6n_{11} - 5 \}, 
    \{ 6n_{11} + 2, 6n_{11} + 8, \ldots, 6n_{11} + 6n_{12} - 4 \}, 
    \\
  \label{basePtnPrimitivemod3}
    & \{ 6n_{11} + 6n_{12} + 3, 6n_{11} + 6n_{12} + 9, \ldots, 6n_{11} + 6n_{12} + 6n_{10} - 3 \}.  
  \end{align}
  The only function of curly braces is to indicate that there are streaks of 
  parts that are the same $\pmod{3}$.  
  Notice that in \eqref{basePtnPrimitivemod3}, the $i \pmod{3}$ singletons 
  are tight in themselves for $i = 1, 2, 0$.  
  None except the smallest one could move further backward.  
  Recall that the singletons are moved thrice at a time, so their value $\pmod{3}$ does not change.  
  The weight of the partition \eqref{basePtnPrimitivemod3} is 
  \begin{align}
  \nonumber
    & 3 n_{11}^2 - 2 n_{11} + 6 n_{12} n_{11} + 3 n_{12}^2 - n_{12} 
	+ 6 n_{11} n_{10} + 6 n_{12} n_{10} + 3 n_{10}^2 
    \\ 
  \label{basePtnPrimitivemod3Weight}
    & = 3 n_{11}^2 - 2 n_{11} + 3 n_{12}^2 - n_{12} + 3 n_{10}^2 
      + 6 n_{12} n_{11} + 6 n_{12} n_{10} + 6 n_{11} n_{10} 
  \end{align}
  Now, observe that the smallest $0 \pmod{3}$ singleton may be triple moved backward twice 
  (i.e. we can subtract 6 from it) with adjustments to follow.  
  We highlight the moved part as usual.   
  \begin{align*}
    & 1, 7, \ldots, 6n_{11} - 5, 
     6n_{11} + 2, 6n_{11} + 8, \ldots, 6n_{11} + 6n_{12} - 10, 
    \\
    & \underbrace{6n_{11} + 6n_{12} - 4, \mathbf{6n_{11} + 6n_{12} - 3}}_{!}, 
    6n_{11} + 6n_{12} + 9, \ldots, 6n_{11} + 6n_{12} + 6n_{10} - 3 
  \end{align*}
  \begin{align*}
   \Bigg\downarrow \textrm{ adjustment }
  \end{align*}
  \begin{align}
  \nonumber
    & 1, 7, \ldots, 6n_{11} - 5, 
     6n_{11} + 2, 6n_{11} + 8, \ldots, 6n_{11} + 6n_{12} - 10, 
    \\
    \label{basePtnPrimitivemod3interm}
    & \mathbf{6n_{11} + 6n_{12} - 6}, 6n_{11} + 6n_{12} - 1, 
    6n_{11} + 6n_{12} + 9, \ldots, 6n_{11} + 6n_{12} + 6n_{10} - 3 
  \end{align}
  The smallest $0 \pmod{3}$ part can be triple moved further backward in a similar fashion 
  for each of the remaining smaller $2 \pmod{3}$ singletons, yielding
  \begin{align*}
    & 1, 7, \ldots, 6n_{11} - 5, 
     6n_{11} + 2, \mathbf{6n_{11} + 6}, 6n_{11} + 11, 6n_{11} + 17, \ldots, 
    \\
    & 6n_{11} + 6n_{12} - 1, 6n_{11} + 6n_{12} + 9, 6n_{11} + 6n_{12} + 15, 
    \ldots, 6n_{11} + 6n_{12} + 6n_{10} - 3.  
  \end{align*}
  In fact, we can perform further triple backward moves on the smallest $0 \pmod{3}$ part, 
  but for reasons we will see below, we hold its current position.  
  
  In total, we performed $n_{12}$ triple backward moves on the smallest $0 \pmod{3}$ singleton.  
  Now, each of these moves allows exactly one triple backward move 
  on each of the larger $0 \pmod{3}$ singletons, 
  yielding the intermediate partition
  \begin{align*}
    1, 7, \ldots, 6n_{11}-5, \mathbf{6n_{11} + 2}, 6n_{11} + 6, 6 n_{11} + 11, 6 n_{11} + 15, \ldots, 
  \end{align*}
  intertwining the $2 \pmod{3}$ and $0 \pmod{3}$ singletons in a tight configuration.  
  This operation reduces the weight of \eqref{basePtnPrimitivemod3} by $3n_{12}n_{10}$.  
  
  We now follow the same path for the smallest $2 \pmod{3}$ singleton $6n_{11} + 2$ 
  and triple move it backward $n_{11}$ times to obtain the configuration
  \begin{align*}
   1, \mathbf{5}, 10, 16, \ldots, 6 n_{11} - 2, 6 n_{11} + 6, 6 n_{11} + 11, 6 n_{11} + 15, \ldots.  
  \end{align*}
  This allows each of the larger $2 \pmod 3$ and $0 \pmod 3$ singletons 
  to be triple moved backward $n_{11}$ times to obtain the configuration 
  \begin{align*}
   \beta = 1, 5, 9, 13, \ldots
  \end{align*}
  with weight 
  \begin{align*}
   3 n_{11}^2 - 2 n_{11} + 3 n_{12}^2 - n_{12} + 3 n_{10}^2 
   + 3 (n_{11}n_{12} + n_{11}n_{10} + n_{12}n_{10}).  
  \end{align*}
  No further triple moves backward is possible for $\beta$.  
  The reason why we did not want the smallest $0 \pmod 3$ singleton further backward 
  in \eqref{basePtnPrimitivemod3interm} is clear now.  
  In the base partition $\beta$ we reached in the end, the smallest $0 \pmod 3$ 
  singleton is larger than the smallest of the other kinds of singletons.  
  
  Incorporation of $1 \pmod 3$ pairs is straightforward: 
  \begin{align*}
   [1, 4], [7, 11], \ldots, [6 n_{21} - 5, 6 n_{21} - 2],
   6 n_{21} + 1, 6 n_{21} + 5, 6 n_{21} + 9, \ldots.  
  \end{align*}
  The inserted pairs have weight $6n_{21}^2 - n_{21}$, 
  and the singletons are triple moved forward twice $n_{21}$ times, 
  adding $6n_{21}(n_{11} + n_{12} + n_{10})$ to the weight.  
  Incorporation of the $2 \pmod 3$ pairs is performed likewise, 
  concluding the proof.  
\end{proof}

{\bf Example: } Here, we will instantiate the construction of the base partition $\beta$ 
with $n_{11} = 2$, $n_{12} = 3$ and $n_{10} = 4$.  
As in \eqref{basePtnPrimitivemod3}, the primitive of the base partition $\beta$ is 
\begin{align*}
 \{ 1, 7 \}, \{ 14, 20, 26 \}, \{ 33, 39, 45, 51 \}, 
\end{align*}
with weight
\begin{align*}
 & 6 n_{21}^2 - n_{21} + 6 n_{22}^2 + n_{22} + 3 n_{11}^2 - 2 n_{11} + 3 n_{12}^2 - n_{12} + 3 n_{10}^2 \\
 & + 12 n_{21} n_{22} + 6 (n_{21} + n_{22})(n_{11} + n_{12} + n_{10})
 + 6 n_{11} n_{12} + 6 n_{11} n_{10} + 6 n_{12} n_{10} = 236.  
\end{align*}
The weight of the base partition $\beta$ we will obtain in the end is expected to be
\begin{align*}
 & 6 n_{21}^2 - n_{21} + 6 n_{22}^2 + n_{22} + 3 n_{11}^2 - 2 n_{11} + 3 n_{12}^2 - n_{12} + 3 n_{10}^2 \\
 & + 12 n_{21} n_{22} + 6 (n_{21} + n_{22})(n_{11} + n_{12} + n_{10})
 + 3 n_{11} n_{12} + 3 n_{11} n_{10} + 3 n_{12} n_{10} = 158.  
\end{align*}
In other words, we are anticipating $n_{11} n_{12} + n_{11} n_{10} + n_{12} n_{10}$ 
$= n_{11} n_{12} + (n_{11} + n_{12}) n_{10}$ triple backward moves on the 
base partition primitive.  
Of these moves, $n_{11} + n_{12}$ will be on the $0 \pmod{3}$ singletons, 
and $n_{11}$ of them will be on each of the $2 \pmod{3}$ singletons.  

We start with the $0 \pmod{3}$ singletons, 
and perform $n_{21}$ triple backward moves on each of them, 
from the smallest to the largest.  
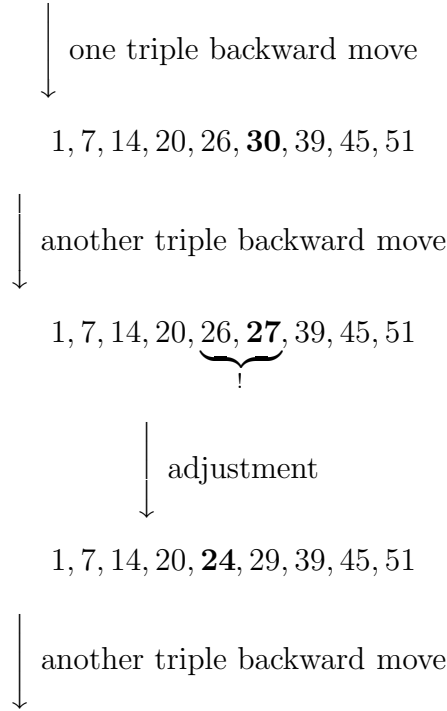
\begin{align*}
 \Bigg\downarrow \textrm{ one triple backward move }
\end{align*}
\begin{align*}
 1, 7, 14, 20, 26, \mathbf{30}, 39, 45, 51
\end{align*}
\begin{align*}
 \Bigg\downarrow \textrm{ another triple backward move }
\end{align*}
\begin{align*}
 1, 7, 14, 20, \underbrace{26, \mathbf{27}}_{!}, 39, 45, 51
\end{align*}
\begin{align*}
 \Bigg\downarrow \textrm{ adjustment }
\end{align*}
\begin{align*}
 1, 7, 14, 20, \mathbf{24}, 29, 39, 45, 51
\end{align*}
\begin{align*}
 \Bigg\downarrow \textrm{ another triple backward move }
\end{align*}
\begin{align*}
 1, 7, 14, \underbrace{20, \mathbf{21}}_{!}, 29, 39, 45, 51
\end{align*}
\begin{align*}
 \Bigg\downarrow \textrm{ adjustment }
\end{align*}
\begin{align*}
 1, 7, 14, \mathbf{18}, 23, 29, 39, 45, 51
\end{align*}
Although we can, we do not move the smallest $0 \pmod{3}$ singleton any further backward.  
We want it to stay larger than the smallest $2 \pmod{3}$ singleton.  
At this point, each of the larger $0 \pmod{3}$ parts can be triple moved backward three times, 
yielding
\begin{align*}
 1, 7, \mathbf{14}, 18, 23, 27, 32, 36, 42.  
\end{align*}
Except for the indicated smallest $2 \pmod{3}$ singleton and the smallest $0 \pmod{3}$ singleton, 
none of the singletons can be moved further back thrice without changing the number of singletons 
or violating the distance 4 condition between singletons.  
We have performed $n_{12} n_{10} = 12$ triple backward moves, 
reducing the weight of the primitive partition by 36.  

Now we will triple move the smallest $2 \pmod{3}$ singleton backward twice, 
because there are two $1 \pmod 3$ singletons smaller than it.  
\begin{align*}
 \Bigg\downarrow \textrm{ a triple backward move }
\end{align*}
\begin{align*}
 1, 7, \mathbf{11}, 18, 23, 27, 32, 36, 42  
\end{align*}
\begin{align*}
 \Bigg\downarrow \textrm{ another triple backward move }
\end{align*}
\begin{align*}
 1, \underbrace{7, \mathbf{8}}_{!}, 18, 23, 27, 32, 36, 42  
\end{align*}
\begin{align*}
 \Bigg\downarrow \textrm{ adjustment }
\end{align*}
\begin{align*}
 1, \mathbf{5}, 10, 18, 23, 27, 32, 36, 42  
\end{align*}
Now we know that all of the larger $2 \pmod{3}$ and $0 \pmod{3}$ singletons 
can be triple moved backward twice, ending the construction.  
\begin{align*}
 1, 5, \underbrace{10, 12}_{!}, 17, 21, 26, 30, 36  
\end{align*}
\begin{align*}
 \Bigg\downarrow \textrm{ adjustment }
\end{align*}
\begin{align*}
 \beta = 1, 5, 9, 13, 17, 21, 26, 30, 36  
\end{align*}
We have made $n_{11}(n_{12} + n_{10}) = 14$ more triple moves backward, 
further reducing the weight by 42.  
Notice that $\beta$ above is not a base partition in the sense of the proof of Theorem \ref{thmMain}, 
but it is a base partition if we require specific numbers of $i \pmod{3}$ singletons for $i = 1, 2, 0$.  
No singleton in $\beta$ can be triple moved backward without reducing the number of singletons.  

If we combine Theorem \ref{thmMod3}, \cite[(2.10)]{Alladi-Gordon}, 
and~\cite[(2.8)]{AAB}, we obtain the following.  

\begin{corollary}
\label{corAlladiAndrewsBerkovich}
\begin{align}
\label{eqCorAlladiAndrewsBrekovich}
    \nonumber
      & \sum_{ \substack{n_{11}, n_{12}, n_{10} \geq 0 \\ n_{21}, n_{22} \geq 0} } \frac
      { q^{ 6 n_{21}^2 - n_{21} + 6 n_{22}^2 + n_{22} + 3 n_{11}^2 - 2 n_{11} + 3 n_{12}^2 - n_{12} + 3 n_{10}^2 } }
      { (q^3; q^3)_{n_{11}} (q^3; q^3)_{n_{12}} (q^3; q^3)_{n_{10}} (q^{6}; q^{6})_{n_{21}} (q^{6}; q^{6})_{n_{22}} }  \\
      \nonumber
      & \times q^{ 12 n_{21} n_{22} + 6 (n_{21} + n_{22})(n_{11} + n_{12} + n_{10})
		  + 3 n_{11} n_{12} + 3 n_{11} n_{10} + 3 n_{12} n_{10}} \\ 
    \nonumber
      & \times a^{ 2 n_{21} + n_{11} + n_{10} } 
      b^{ 2 n_{22} + n_{12} + n_{10} }  \\[10pt]
      & = (-aq;q^3)_\infty (-bq^2; q^3)_\infty.  
\end{align}

\end{corollary}

\section{Comments and Further Problems}

It is possible to construct series for partitions 
satisfying Schur's condition and having smallest part at least 2, 3 or 4 (cf.~\cite{A-Schur}), 
or generating functions for partitions in which parts are at least three apart
and parts that are 1 or 2 $\pmod{3}$, instead of multiples of 3, are at least six apart.  

One open question is the investigation of a proof of Schur's identity and several companions to it 
in the spirit of Andrews' analytic proof of Rogers-Ramanujan-Gordon identities~\cite{PNAS}.  

It must be possible to construct series as generating functions of partitions 
in Andrews' generalization of Schur's theorem~\cite{A-Schur-Gnl, A-Ptn-Diff}.  

Another open problem is the construction of series as generating functions of 
overpartitions~\cite{CL-Overptn} in Lovejoy's and Dousse's extensions of Schur's theorem to 
overpartitions~\cite{Lovejoy-Schur-Overptn, Dousse-Schur1, Dousse-Schur2, Dousse-Schur1, Dousse-Schur2}.  
There is also a combinatorial proof of results from~\cite{Lovejoy-Schur-Overptn} 
in~\cite{Rag-Pad}.  

% The extension of Schur's theorem to overpartitions goes back to my paper, 
% A theorem on seven-colored overpartitions and its applications. 
% Int. J. Number Theory 1 (2005), no. 2, 215--224. MR2173381 .    
% Dousse gave a couple of new proofs and generalized it to an arbitraary number of products.    
% (Padmavathamma and Raghavendra R. gave a combinatorial proof in:  
% Bressoud's generalization of Schur's theorem extension to overpartitions. 
% New Zealand J. Math. 39 (2009), 25--32. MR2646994)

\section*{Acknowledgements}
We thank George E. Andrews for pointing out the questions that are answered 
in sections \ref{secMod2} and \ref{secMod3}; 
and Jehanne Dousse for useful discussions during the preparation of this manuscript.  
We also thank Jeremy Lovejoy for notifying us of the incorrect statement of 
Corollary \ref{corAlladiAndrewsBerkovich} stemming from a mistake in 
the statement of Theorem \ref{thmMod3} and 
for completing the list of references for Schur's theorem 
for overpartitions~\cite{Lovejoy-Schur-Overptn, Rag-Pad}


\begin{thebibliography}{99}

\bibitem{Alladi-Gordon} Alladi, K., Gordon, B., 
Schur's partition theorem, companions, refinements and generalizations, 
\emph{Trans. AMS}, {\bf 347(5)}, 1591--1608, 1995.  

\bibitem{AAB} Alladi, K., Andrews, G.E., Berkovich, A., 
A new four-parameter $q$-series identity and its partition implications, 
\emph{Invent. Math.}, {\bf 153}, 231--260, 2003.  

\bibitem{A-Schur} Andrews, G.E., 
On partition functions related to Schur's second partition theorem. 
\emph{Proc. Amer. Math. Soc.}, {\bf 18}, 441--444, 1968.  

\bibitem{A-Schur-Gnl} Andrews, G.E., 
A new generalization of Schur's second partition theorem. 
\emph{Acta Arith.}, {\bf 14}, 429--434, 1968.  

\bibitem{A-Ptn-Diff} Andrews, G.E., 
A general theorem on partitions with difference conditions. 
\emph{Amer J. Math}, {\bf 91}, 18--24, 1969.  

\bibitem{PNAS} Andrews, G.E., 
An analytic generalization of the Rogers-Ramanujan identities for odd moduli, 
\emph{Proc. Nat. Acad. Sci. USA}, {\bf 71}, 4082--4085, 1974.  

\bibitem{TheBlueBook} Andrews, G.E., 
\emph{The Theory of Partitions}, 
The Encyclopedia of Mathematics and Its Applications Series, Addison-Wesley Pub. Co., NY, (1976). 
Reissued, Cambridge University Press, New York, 1998.  

\bibitem{ABM} Andrews, G.E., Bringmann, K., Mahlburg, K., 
Double series representations for Schur's partition function and related identities, 
\emph{JCT A}, {\bf 132}, 102--119, 2015.  

\bibitem{AE} Andrews, G.E., Eriksson, K., 
\emph{Integer Partitions},
Cambridge University Press, Cambridge, UK, (2004).  

\bibitem{Bres} Bressoud, D.M., 
An analytic generalization of the Rogers-Ramanujan identities with interpretation
\emph{Quart. J. Oxford.}, {\bf 31}, 385--399, 1981.  

\bibitem{CL-Overptn} Corteel, S., Lovejoy, J., 
Overpartitions, 
\emph{Trans. Amer. Math. Soc.}, {\bf 356}, 1623--1635, 2004.

\bibitem{Dousse-Schur1} Dousse, J., 
On generalizations of partition theorems of Schur and Andrews to overpartitions, 
\emph{The Ramanujan J.}, {\bf 35(3)}, 339--360, 2014.  

\bibitem{Dousse-Schur2} Dousse, J., 
Unification, Refinements and Companions of Generalisations of Schur’s Theorem. 
In: Andrews G., Garvan F. (eds) \emph{Analytic Number Theory, Modular Forms and q-Hypergeometric Series. 
ALLADI60 2016}. Springer Proceedings in Mathematics \& Statistics, vol 221. Springer, Cham, 2017.  

\bibitem{Euler} Euler, L.,
\emph{Introduction  to  Analysis  of  the  Infinite}, 
transl. by J. Blanton, Springer, New York, 1988.  

\bibitem{KK-parity} Kur\c{s}ung\"{o}z, K., 
Parity considerations in Andrews-Gordon identities, 
\emph{Eur. J. Comb.}, {\bf 31(3)}, 976--1000, 2010.  

\bibitem{KK-Cpr} Kur\c{s}ung\"{o}z, K., 
Andrews-Gordon Type Series for Capparelli's and G\"{o}llnitz-Gordon Identities, 
\emph{preprint}, 
{\tt https://arxiv.org/abs/1807.11189}.  

\bibitem{Lovejoy-Schur-Overptn} Lovejoy, J., 
A theorem on seven-colored overpartitions and its applications, 
\emph{Int. J. Number Theory} {\bf 1(2)}, 215--224, 2005.    

\bibitem{Rag-Pad} Padmavathamma, Raghavendra R., 
Bressoud's generalization of Schur's theorem extension to overpartitions. 
\emph{New Zealand J. Math.} {\bf 39}, 25--32, 2009.  


\bibitem{RR} Ramanujan, S., Rogers, L.J., 
Proof of certain identities in combinatory analysis, 
\emph{Proc. Cambridge. Phil. Soc.}, {\bf 19}, 211--216, (1919).  

\bibitem{Schur} Schur, I. 
In Zur Additiven Zahlentheorie, 
\emph{Sitzungsberichte der Preussischen Akademie der Wissenschaften}, 
488--495, Berlin, 1926.  

\end{thebibliography}
\end{document}